\documentclass[11pt]{article}
\usepackage{a4wide}

\usepackage{booktabs}
\usepackage[ruled,vlined,linesnumbered]{algorithm2e}
\usepackage{amsthm}
\usepackage{amsmath}
\usepackage{amssymb}
\usepackage{url}

\newtheorem{thm}{Theorem}
\newtheorem{lem}{Lemma}
\newtheorem{conj}{Conjecture}
\newtheorem{rem}{Remark}

\title{On consecutive primitive elements in a finite field}

\author{
  Stephen D. Cohen \\
  School of Mathematics and Statistics, \\
  University of Glasgow, Scotland \\
  Stephen.Cohen@glasgow.ac.uk
\and
  Tom\'{a}s Oliveira e Silva \\
  Departamento de Electr{\'o}nica, Telecomunica{\c c}{\~o}es e Inform{\'a}tica / IEETA \\
  University of Aveiro, Portugal \\
  tos@ua.pt
\and
  Tim Trudgian\footnote{Supported by Australian Research Council DECRA Grant DE120100173.} \\
  Mathematical Sciences Institute \\
  The Australian National University, ACT 0200, Australia \\
  timothy.trudgian@anu.edu.au
}

\date{October 20, 2014}

\begin{document}

\maketitle

\begin{abstract}
  \noindent
  For $q$ an odd prime power with $q>169$ we prove that there are always three consecutive
  primitive elements in the finite field $\mathbb{F}_{q}$. Indeed, there are precisely eleven  values
  of $q \leq 169$  for which this is false.  For $4\leq n \leq 8$ we present
  conjectures on the size of $q_{0}(n)$ such that $q>q_{0}(n)$ guarantees the existence of $n$
  consecutive primitive elements in $\mathbb{F}_{q}$, provided that $\mathbb{F}_{q}$ has
  characteristic at least~$n$. Finally, we improve the upper bound on $q_{0}(n)$ for all
  $n\geq 3$.
\end{abstract}


\section{Introduction}

Let $q$ be a prime power and consider primitive elements in $\mathbb{F}_{q}$, the finite field of
order~$q$. Cohen~\cite{Cohen2,Cohen:Last,Cohen:1985} proved that $\mathbb{F}_{q}$ contains two consecutive
distinct primitive elements whenever~$q>7$. For $n\geq 2$ we wish to determine $q_{0}(n)$ such
that $\mathbb{F}_{q}$, assumed to have characteristic larger than or equal to $n$, contains $n$
consecutive distinct primitive elements for all~$q>q_{0}(n)$.

Carlitz~\cite{Carlitz} showed that $q_{0}(n)$ exists for all~$n$. Tanti and
Thangadurai~\cite[Thm.\ 1.3]{Tanti} showed that
\begin{equation}\label{tanti}
  q_{0}(n) \leq \exp(2^{5.54 n}), \quad (n\geq 2).
\end{equation}
When $n=3$ this gives the enormous bound~$10^{43743}$. The main point of this article is to apply
techniques from~\cite{Cohen:2014} to prove
\begin{thm}\label{water}
  The finite field $\mathbb{F}_{q}$ contains three consecutive primitive elements for all odd
  $q>169$.  Indeed, the only fields $\mathbb{F}_{q}$ (with $q$ odd) which do \emph{not} contain three consecutive primitive elements
  are those for which $q=$ $3$, $5$, $7$, $9$, $13$, $25$, $29$, $61$, $81$, $121$, or $169$.
\end{thm}
When $n\geq 4$ we improve the estimate for $q_{0}(n)$ in
\begin{thm}\label{rum}
  The field $\mathbb{F}_{q}$, assumed to have characteristic at least $n$, contains $n$
  consecutive primitive elements provided that $q>q_{0}(n)$, where values of $q_{0}(n)$ are given
  in the third column of Table~$\ref{Surletable}$ for $4\leq n \leq 10$ and
  $q_{0}(n)=\exp(2^{2.77 n})$ for $n\geq 11$.
\end{thm}

\begin{table}[ht]
  \caption{Range of potential $q$ given a value of $n$}
  \label{Surletable}
  \centering
  \begin{tabular}{rrl}
    \hline\hline
     $n$ & Bound on $\omega(q-1)$ & Bound on $q$ ($q_0(n)$)            \\
    \hline\noalign{\vspace*{2pt}}
     $3$ &        $13\mkern 60mu$ & $\mkern 25mu   3.49\times 10^{15}$ \\
     $4$ &        $23\mkern 60mu$ & $\mkern 25mu   3.29\times 10^{32}$ \\
     $5$ &        $37\mkern 60mu$ & $\mkern 25mu   4.22\times 10^{61}$ \\
     $6$ &        $59\mkern 60mu$ & $\mkern 25mu  4.61\times 10^{113}$ \\
     $7$ &       $100\mkern 60mu$ & $\mkern 25mu  3.75\times 10^{220}$ \\
     $8$ &       $171\mkern 60mu$ & $\mkern 25mu  2.27\times 10^{425}$ \\
     $9$ &       $301\mkern 60mu$ & $\mkern 25mu  1.01\times 10^{836}$ \\
    $10$ &       $533\mkern 60mu$ & $\mkern 25mu 6.69\times 10^{1638}$ \\
    \hline
  \end{tabular}
\end{table}
We prove Theorem~\ref{rum} and discuss the construction of Table~\ref{Surletable} in
Section~\ref{brandy}.

We remark that the outer exponent in the bound in Theorem~\ref{rum} is half of that given
in~(\ref{tanti}) owing entirely to the superior sieving inequality used in Theorem~\ref{basicthm}.
The double exponent still gives an enormous bound on $q_{0}(n)$. Were one interested in bounds for
specific values of $n\geq 11$ one should extend Table \ref{Surletable} as per
Section~\ref{brandy}.

In Section~\ref{algo} we present an algorithm that, along with Theorem~\ref{rum}, proves Theorem~\ref{water}.

Whereas we are not able to resolve completely the values of $q_{0}(n)$ for $n\geq 4$, we present,
in Section~\ref{conj}, some conjectures as to the size of $q_{0}(n)$ for $4\leq n \leq 8$.

\section{Character sum expressions and estimates}

Let $\omega(m)$ denote the number of distinct prime factors of $m$ so that $W(m)=2^{\omega(m)}$ is
the number of square-free divisors of~$m$. Also, let $\theta(m)=\prod_{p|m}(1-p^{-1})$. For any
integer~$m$ define its radical $\mathrm{Rad}(m)$ as the product of all distinct prime factors
of~$m$.

Let $e$ be a divisor of~$q-1$. Call $g \in \mathbb{F}_{q}$ \emph{$e$-free} if $g \neq 0$ and
$g = h^d$, where $h \in \mathbb{F}_{q}$ and $d|e$ implies $d=1$. The notion of $e$-free depends
(among divisors of $q-1$) only on~$\mathrm{Rad}(e)$. Moreover, in this terminology a primitive
element of $\mathbb{F}_{q}$ is a $(q-1)$-free element.

The definition of any multiplicative character $\chi$ on $\mathbb{F}_{q}^*$ is extended to the
whole of $\mathbb{F}_{q}$ by setting $\chi(0)=0$. In fact, for any divisor $d$ of $\phi(q-1)$,
there are precisely $\phi(d)$ characters of order (precisely) $d$, a typical such character being
denoted by~$\chi_d$. In particular, $\chi_1$, the principal character, takes the value $1$ at all
non-zero elements of $\mathbb{F}_{q}$ (whereas $\chi_1(0)=0$). A convenient shorthand notation to
be employed for any divisor $e$ of $q-1$ is
\begin{equation}\label{ale}
  \int_{d|e} = \sum_{d|e}\frac{\mu(d)}{\phi(d)} \sum_{\chi_{d}},
\end{equation}
where the sum over $\chi_{d}$ is the sum over all $\phi(d)$ multiplicative characters $\chi_d$ of
$\mathbb{F}_{q}$ of exact order~$d$. Its significance is that, for any $g \in \mathbb{F}_{q}$,
\[
  \theta(e)\int_{d|e}\chi_d(g) =
  \left\{ \begin{array}{cl}
    1, & \mbox{if $g$ is non-zero and $e$-free}, \\
    0, & \mbox{otherwise.}
  \end{array} \right.
\]
In this expression (and throughout) only characters $\chi_d$ with $d$ square-free contribute (even
if $e$ is not square-free).

The present investigation concerns the question of the existence of $n \geq 2$ consecutive
\emph{distinct} primitive elements in $\mathbb{F}_{q}$, i.e., whether there exists
$g \in \mathbb{F}_{q}$ such that $\{g, g+1,\ldots, g+n-1\}$ is a set of $n$ distinct primitive
elements of~$\mathbb{F}_{q}$. Define $p$ to be the characteristic of $\mathbb{F}_q$, so
that $q$ is a power of $p$.   Then, necessarily, $n \leq p$ and, by Theorem~1
of~\cite{Cohen:2014}, we can suppose~$n \geq 3$. Assume therefore throughout that
$3 \leq n\leq p$. In particular, $q$ is odd.

Let $e_1, \ldots, e_n$ be divisors of $q-1$. (In practice all divisors will be \emph{even}.)
Define $N(e_1, \ldots, e_n)$ to be the number of (non-zero) $g \in \mathbb{F}_{q}$ such that
$g+k-1$ is $e_k$-free for each $k=1, \ldots, n$. The first step is the standard expression for
this quantity in terms of the multiplicative characters of~$\mathbb{F}_{q}^*$.
\begin{lem}\label{Nexpr}
  Suppose $3 \leq n\leq p$ and $e_1, \ldots, e_n$ are divisors of $q-1$. Then
  \[
     N(e_1, \ldots, e_n)= \theta(e_1)\ldots\theta(e_n)
       \int_{d_1|e_1}\ldots \int_{d_n|e_n}S(\chi_{d_1}, \ldots,\chi_{d_n}),
  \]
  where
  \[
    S(\chi_{d_1}, \ldots,\chi_{d_n})=
      \sum_{g \in \mathbb{F}_{q}}\chi_{d_1}(g) \ldots \chi_{d_n}(g+n-1).
  \]
\end{lem}
We now provide a bound on the size of $S(\chi_{d_1}, \ldots,\chi_{d_n})$.
\begin{lem}\label{weil}
  Suppose $d_1, \ldots, d_n$ are square-free divisors of $q-1$. Then
  \[
    S(\chi_{d_1}, \ldots,\chi_{d_n})=q-n \quad\mbox{if $d_1= \cdots= d_n=1$,}
  \]
  and otherwise
  \[
    \bigl|S(\chi_{d_1}, \ldots,\chi_{d_n})\bigr| \leq (n-1)\sqrt{q}.
  \]
\end{lem}

\begin{proof}
  We can assume not all of $d_1, \ldots, d_n$ have the value $1$. Then
  $d:= \mathrm{lcm}(d_1, \ldots, d_n)$ is also a square-free divisor of $q-1$ and  $d>1$.
  Evidently, there are positive integers $c_1,\ldots,c_n$ with $\gcd(c_k,d_k)=1$ such that
  \[
    S(\chi_{d_1}, \ldots,\chi_{d_n})= \sum_{g \in \mathbb{F}_{q}}\chi_{d}\bigl(f(g)\bigr),
  \]
  where $f(x) =x^{c_1}(x+1)^{c_2} \cdots(x+n-1)^{c_n}$. Since the radical of the polynomial $f$
  has degree $n$ the result holds by Weil's theorem (see Theorem 5.41
  in~\cite[page 225]{Lidl:1997} and also~\cite{Peng:2014}).
\end{proof}

When $e_1= e_2=\cdots =e_n=e$, say, we shall abbreviate $N(e, \ldots,e)$ to $N_n(e)$. We obtain a
lower bound for $N_{n}(e)$ in
\begin{lem}\label{NEst}
  Suppose  $3 \leq n\leq p$ and $e$ is a divisor of $q-1$. Then
  \[
    N_n(e) \geq \theta(e)^n \bigl(q-(n-1)W(e)^n\sqrt{q}\bigr).
  \]
\end{lem}
\begin{proof}
  The presence of the M\"obius function in the integral notation in~(\ref{ale}) means that we need
  only concern ourselves with the square-free divisors $(d_{1}, \ldots, d_{n})$ of~$e$. The
  contribution of each $n$-tuple is given by Lemma~\ref{weil}. Hence,
  $N_n(e) \geq \theta(e)^n \bigl(q-n-(n-1)(W(e)^n-1)\sqrt{q}\bigr)$ and the result follows since
  $n\geq 3$ and $q \geq3$.
\end{proof}


Applying Lemma~\ref{NEst} with $e=q-1$ gives the basic criterion which guarantees $n$ consecutive
primitive elements for sufficiently large~$q$.
\begin{thm}\label{basicthm}
  Suppose $3 \leq n\leq p$. Suppose
  \begin{equation}\label{porter}
    q \geq (n-1)^2\,W(q-1)^{2n}=(n-1)^2\,2^{2n\omega(q-1)}.
  \end{equation}
  Then there exists a set of $n$ consecutive primitive elements in $\mathbb{F}_{q}$.
\end{thm}

As an application of Theorem~\ref{basicthm}, consider the case~$n=3$. Let $P_m$ be the product of
the first $m$ primes. A quick numerical computation reveals that for $m\geq 50$ one has
$P_m+1 \geq 2^{2+6m}$. Hence, it follows that $\mathbb{F}_q$ has three consecutive primitive
elements when $\omega(q-1)\geq 50$ or when $q\geq 2^{2+6\times 50}$. We shall soon improve this
markedly.

We now briefly present an improvement in the above discussion when $q \equiv 3\pmod 4$. The
improvement in this case, in which $-1$ is a non-square in $\mathbb{F}_{q}$, is related to a
device used in~\cite{Cohen:1985}. Note that, in Lemma~\ref{Nexpr}, in the definition of
$S(\chi_{d_1}, \ldots,\chi_{d_n})$ we may replace $g$ by $-g-(n-1)$ and obtain
\begin{align*}
  S(\chi_{d_1}, \ldots,\chi_{d_n}) & =
    \sum_{g \in \mathbb{F}_{q}} \prod_{i=1}^n \chi_{d_i}(-g-n+1+i-1)=
    \sum_{g \in \mathbb{F}_{q}} \prod_{i=1}^n \chi_{d_i}(-1)\chi_{d_i}(g+n-i) \\
    & = (\chi_{d_1}\chi_{d_2}\ldots \chi_{d_n})(-1) S(\chi_{d_n}, \ldots,\chi_{d_1}).
\end{align*}
In particular, if an \emph{odd} number of the integers $d_1, \ldots, d_n$ are even, then
$S(\chi_{d_n}, \ldots,\chi_{d_1})=-S(\chi_{d_1}, \ldots,\chi_{d_n})$. As a consequence, in
Lemma~\ref{Nexpr}, if $e_1= \cdots =e_n=e$ is even, then, on the right side of the expression for
$N_n(e)$, the terms corresponding to divisors $(d_1, \ldots, d_n)$ with an odd number of even
divisors cancel out exactly and only those with an \emph{even} number of even divisors contribute.
This accounts for precisely half the $n$-tuples $(d_1, \ldots, d_n)$. Accordingly, we obtain the
following improvement of Lemma~\ref{NEst} and Theorem~\ref{basicthm} in this situation.
\begin{thm}\label{NEst1}
  Suppose $3\leq n\leq p$ and $e$ is an even divisor of $q-1$, where $q\equiv 3 \pmod 4 , (q \geq 7)$.
  Then
  \[
    N_n(e) \geq \theta(e)^n \Bigl(q- \frac{n-1}{2}\,W(e)^n\sqrt{q}\Bigr).
  \]
  Moreover, if
  \[
    q \geq \frac{(n-1)^2}{4}W(q-1)^{2n}= (n-1)^2 2^{2(n\omega(q-1)-1)},
  \]
  then there exists a set of $n$ consecutive primitive elements in $\mathbb{F}_{q}$.
\end{thm}

\section{Sieving inequalities and estimates}

As before assume $3 \leq n \leq p$ and let $e$ be a divisor of $q-1$. Given positive integers
$m, j, k$ with $1 \leq j,k \leq n$ define $m_{jk}=m$ if $j=k$ and otherwise $m_{jk}=1$.
\begin{lem}\label{Nle}
  Suppose $3 \leq n\leq p$ and $e$ is a divisor of $q-1$. Let $l$ be a prime divisor of $q-1$ not
  dividing~$e$. Then, for each $j \in \{1, \ldots, n\}$,
  \[
    \bigl|N(l_{j_1}e, \ldots l_{jn}e) -\theta(l)N_n(e)\bigr| \leq
    (1-1/l)\,\theta(e)^{n}(n-1)W(e)^n \sqrt{q}.
  \]
\end{lem}
\begin{proof}
  Observe that $\theta(le)= (1-1/l)\,\theta(e)$ and that, by Lemma~\ref{Nexpr}, all the character
  sums in $\theta(l)N_n(e)$ appear identically in $N(l_{j_1}e, \ldots l_{jn}e)$. Hence, by
  Lemma~\ref{weil},
  \[
    \bigl|N(l_{j_1}e, \ldots l_{jn}e) -\theta(l)N_n(e)\bigr| \leq
    (1-1/l)\,\theta(e)^{n}(n-1)W(e)^{n-1}\,\bigl(W(le)-W(e)\bigr) \sqrt{q}
  \]
  and the result follows since $W(le)=2W(e)$.
\end{proof}

This leads to the main sieving result. Let $e$ be a divisor of $q-1$. In what follows, if
$\mathrm{Rad}(e)=\mathrm{Rad}(q-1)$ set $s=0$ and $\delta=1$. Otherwise, let $p_1, \ldots, p_s$,
$s \geq 1$, be the primes dividing $q-1$ but not~$e$, and set $\delta=1-n\sum_{i=1}^s p_i^{-1}$.
It is \emph{essential} to choose $e$ so that $\delta >0$.
\begin{lem}\label{sieve}
  Suppose $3 \leq n\leq p$ and $e$ is a divisor of $q-1$. Then, with the above notation,
  \begin{equation}\label{sieveeq1}
    N_n(q-1) \geq
      \Bigl( \sum_{j=1}^n\sum_{i=1}^sN(p_{ij1}e, \ldots, p_{ijn}e) \Bigr) - (ns-1)N_n(e),
  \end{equation}
  where $p_{ijk}$ means $(p_i)_{jk}$. Hence
  \begin{equation}\label{sieveeq2}
    N_n(q-1) \geq  \delta N_n(e) + \sum_{i=1}^s \Bigl( \sum_{j=1}^n
      N(p_{ij1}e, \ldots, p_{ijn}e) -\theta(p_i)N_n(e) \Bigr).
  \end{equation}
\end{lem}
\begin{proof}
  It suffices to suppose $s \geq 1$. The various $N$ terms on the right side of~(\ref{sieveeq1})
  can be regarded as counting functions on the set of $g \in \mathbb{F}_{q}$ for which $g+j-1$ is
  $e$-free for each $j=1, \ldots, n$. In particular, $N_n(e)$ counts all such elements, whereas,
  for each $i=1, \ldots, s$, and $1 \leq j \leq n$, $N(p_{ij1}e,\ldots, p_{ijn}e)$ counts only
  those for which additionally $g+j-1$ is $p_i$-free. Note that $N_n(q-1)$ is the number of
  elements $g$ such that, not only are $g+j-1$ $e$-free for each $j=1, \ldots, n$ but additionally
  $g+j-1$ is $p_i$-free for each $1\leq i\leq s,\ 1 \leq j \leq n $. Hence we see that, for a given
  $g \in \mathbb{F}_{q}$, the right side of~(\ref{sieveeq1}) clocks up $1$ if $g+k-1$ is primitive
  for every $1 \leq k\leq n$, and otherwise contributes a non-positive (integral) quantity. This
  establishes~(\ref{sieveeq1}). Since $\theta(p_i)=1-1/p_i$, the bound~(\ref{sieveeq2}) is deduced
  simply by rearranging the right side of~(\ref{sieveeq1}).
\end{proof}
We are now able to provide a condition that, if satisfied, is sufficient to prove the existence of
$n$ consecutive primitive elements.
\begin{thm}\label{main}
  Suppose $3 \leq n\leq p$ and $e$ is a divisor of $q-1$. If $\mathrm{Rad}(e)=\mathrm{Rad}(q-1)$
  set $s=0$ and $\delta=1$. Otherwise, let $p_1, \ldots, p_s$, $s \geq 1$, be the primes dividing
  $q-1$ but not $e$ and set $\delta=1-n\sum_{i=1}^s p_i^{-1}$. Assume $\delta>0$. If also
  \begin{equation}
    q > \left((n-1)\left(\frac{ns-1}{\delta}+2\right) W(e)^n\right)^2,
    \label{e:main}
  \end{equation}
  then there exist $n$ consecutive primitive elements in $\mathbb{F}_{q}$.
\end{thm}
\begin{proof}
  Assume $\delta>0$. From~(\ref{sieveeq2}) and Lemmas~\ref{NEst} and \ref{Nle} (noting that for
  each $j=1,\ldots, n$ the contribution to (\ref{sieveeq2}) is the same),
  \begin{eqnarray*}
    N_n(q-1) & \geq &
      \theta(e)^n\left(\delta(q-(n-1)W(e)^n\sqrt{q})-
       n\smash{\sum_{i=1}^s} \left(1-\frac{1}{p_i}\right)(n-1)W(e)^n\sqrt{q}\right) \\
    & = & \delta\,\theta(e)^n\sqrt{q}
      \left(\sqrt{q}-(n-1)W(e)^n-(n-1) \left(\frac{ns-1}{\delta}+1\right)W(e)^n\right).
  \end{eqnarray*}
  The conclusion follows.
\end{proof}

We conclude this section with the slight improvement when $q \equiv 3 \pmod 4$. Now, when $e$ is
even, the character sum expressions for the terms
$\sum_{j=1}^n N(p_{ij1}e, \ldots, p_{ijn}e) -\theta(p_i)N_n(e)$ in~(\ref{sieveeq2}) cancel
unless an \emph{even} number of divisors $d_1, \ldots, d_n$ are even. This reduces the bound in
Lemma~\ref{Nle} by a factor of two, and gives the following improvement to Theorem~\ref{main}.

\begin{thm}\label{main1}
  Suppose  $q \equiv 3 \pmod 4 , \ (q \geq 7)$, $3 \leq n\leq p$ and $e$ is an even divisor of $q-1$.
  If $\mathrm{Rad}(e)=\mathrm{Rad}(q-1)$ set $s=0$ and $\delta=1$. Otherwise, let
  $p_1, \ldots, p_s$, $s \geq 1$, be the primes dividing $q-1$ but not $e$ and
  set $\delta=1-n\sum_{i=1}^s p_i^{-1}$. Assume $\delta>0$. If also
  \[
    q > \left(\frac{n-1}{2}\left(\frac{ns-1}{\delta}+2\right) W(e)^n\right)^2,
  \]
  then there exist $n$ consecutive primitive elements in $\mathbb{F}_{q}$.
\end{thm}

\begin{rem}
  In Theorems~$\ref{main}$ and $\ref{main1}$, for a given $s$ the best (largest) value of $\delta$
  is obtained when the largest $s$ prime factors of $q-1$ are used to compute $\delta$.
\end{rem}

\section{\boldmath
  Application of Theorems~\ref{basicthm} and~\ref{main} for generic~$n$;
  proof of Theorem~\ref{rum}}\label{brandy}

As an application of Theorem~\ref{main}, consider the case $n=3$. We showed after
Theorem~\ref{basicthm} that $\mathbb{F}_{q}$ contains three consecutive primitive elements for all
$q$ satisfying $\omega(q-1) \geq 50$. For $14\leq \omega(q-1) \leq 49$ we verify easily
that~(\ref{e:main}) holds with $s=8$. As an example, consider $\omega(q-1)=14$ and $s=8$, whence
$\delta\geq 1-3\bigl(\frac{1}{17}+\frac{1}{19}+\frac{1}{23}+\frac{1}{29}+\frac{1}{31}+\frac{1}{37}+
\frac{1}{41}+\frac{1}{43}\bigr)> 0.1109$. It follows that the right hand side of~(\ref{e:main})
is slightly larger than $12\,03974\,78114\,70119$, which is smaller than $P_{14}$. It follows that
$\mathbb{F}_q$ has three consecutive primitive elements when $\omega(q-1)=14$.

We are unable to proceed directly when $\omega(q-1) \leq 13$. For example, when $\omega(q-1) = 13$
there is no value of $s$ with $1\leq s \leq 13$ which resolves~(\ref{e:main}). If we choose $s=8$
we minimise the right-side of~(\ref{e:main}) whence it follows that we need only consider
$\omega(q-1) \leq 13$ and $q\leq 3.49\times 10^{15}$.

We continue this procedure for larger values of~$n$. We use Theorem~\ref{basicthm} to obtain an
initial bound on $\omega(q-1)$, then use Theorem~\ref{main}, with suitable values of~$s$, to
reduce this bound as far as possible. We therefore reduce the problem of finding $n$ consecutive
primitive elements in $\mathbb{F}_{q}$ to the finite computation in which we need only check those
$q$ in a certain range: this range is given in Table \ref{Surletable}.

We now use Theorem~\ref{basicthm} to obtain a bound on $q_{0}(n)$ for a generic value of~$n$. To
bound $\omega(q-1)$ we use Robin's result~\cite[Thm.\ 11]{Robin}, that
$\omega(n) \leq 1.38402\log n/(\log\log n)$ for all $n\geq 3$. Since the function
$\log x /(\log\log x)$ is increasing for $x\geq e^{e}$ we have
\begin{equation}\label{stout}
  \omega(q-1) \leq \frac{1.38402\log q}{\log\log q},
\end{equation}
for all $q\geq 17$. It is easy to check that~(\ref{stout}) holds also for all $3\leq q \leq 17$.
We use~(\ref{stout}) to rearrange the condition in~(\ref{porter}), showing that
\begin{equation}\label{bitter}
  \log q \left\{ 1 - \frac{2.76804 n \log 2}{\log\log q}\right\} \geq 2 \log (n-1).
\end{equation}
We solve~(\ref{bitter}) by first insisting that the term in braces be bounded below by $d$, where
$d\in(0, 1)$, and then insisting that $d \log q \geq 2\log(n-1)$. This shows that~(\ref{porter})
is certainly true provided that
\begin{equation}\label{lager}
  q \geq \max\{ (n-1)^{2/d}, \exp(2^{2.76804 n/(1-d)})\}.
\end{equation}
We choose $d = 0.0001$, so that we require $q \geq \exp(2^{2.77 n})$ for all $n\geq 6$. This
proves Theorem~\ref{rum}.

We remark that were one to use \cite[Thm.\ 12]{Robin} one could replace the bound in (\ref{stout}) by $ \log q/\log\log q + 1.458 \log q /(\log\log q)^{2}$. This would show, when $n$ is sufficiently large, that the exponent in Theorem~\ref{rum} could be reduced from $2.77$ to $2+ \epsilon$ for any positive $\epsilon$: we have not pursued this.

\section {Three consecutive primitive elements}\label{algo}

To prove Theorem~\ref{water} we verified numerically the existence of three
consecutive primitive elements for all values of $q$ that remained after the
application of Theorem~\ref{main}. As explained in Section~\ref{brandy}, for $n=3$ it is only
necessary to consider the cases where $\omega(q-1)\leq 13$. For each possible value of
$\omega(q-1)$ Theorem~\ref{main} was used to compute a bound on the values of $q$ below which the existence of three consecutive primitive elements was not ensured; these upper bounds
are presented in the second column of Table~\ref{wine}. Algorithm~\ref{beer} was then used to
generate the values of $q$ that required testing.

\begin{algorithm}[hbpt]
  \DontPrintSemicolon
  \caption{Enumeration of all odd integers $m+1$ that satisfy the conditions $m<M$ and
    $\omega(m)=w$.\label{beer}}
  Set $L$ to $\lfloor(M-1)/\prod_{i=1}^{w-1}p_i\rfloor$; here, $p_i$ is the $i$-th prime
    ($p_1=2$, $p_2=3$, and so on) \;
  Generate all tuples $(u_i,v_i)$ of the form $(p^k,p)$, with $p$ an odd prime, $k$ a positive
  integer, and $p^k\leq L$, and sort them in increasing order of the value of $u_i$, so that
  $(u_1,v_1)=(3,3)$, $(u_2,v_2)=(5,5)$, $(u_3,v_3)=(7,7)$, $(u_4,v_4)=(9,3)$, and so on \;
  Append the tuple $(\infty,0)$ to the list of tuples \;
  \For{$k=1,2,\ldots,\lfloor\log(M-1)/\log 2\rfloor$}
  {
    Set $m_1$ to $2^k$, $i_2$ to $0$, and $d$ to $2$ \;
    \eIf{$w=1$}
      {Test $m_1+1$}
      {
        \While{$d>1$}
        {
          \Repeat{$v_{i_d}\not\in\{ v_{i_2},\ldots,v_{i_{d-1}}\}$}{Increment $i_d$}
          \eIf{$m_{d-1}\,u_{i_d}^{w+1-d}\geq M$}
            {Decrement $d$}
            {
              Set $m_d$ to $m_{d-1}u_{i_d}$ \;
              \eIf{$d=w$}
                {Test $m_d+1$}
                {Increment $d$ and then set $i_d$ to $i_{d-1}$}
            }
        }
      }
  }
\end{algorithm}

Lines $10$--$12$ of Algorithm~\ref{beer} ensure that $\gcd(m_{d-1},u_{i_d})=1$ whenever line~$13$
is reached, and ensure that $\omega(m_d)=\omega(m_{d-1})+1$ (i.e., that $\omega(m_d)=d$) every
time line~$17$ is reached. Testing $m+1$, with either $m=m_1$ or $m=m_d$, amounts to verifying that
$m+1$ is an odd prime power. If so, we see whether Theorem~\ref{main} can deal with it; if not, we verify that there exist three consecutive primitive elements in the corresponding finite field.
It turned out to be faster to rule out values of $m+1$ using all possible values of $s$ in
Theorem~\ref{main} (treating $m+1$  as if it were a prime power) before testing if it were a
prime or a prime power.

Algorithm~\ref{beer} was coded using the PARI/GP calculator programming language~\cite{pari}
(version 2.7.2 using a GMP 6.0.0 kernel) and was run for $w=1,2,\ldots,13$ on one core of a
$3.3$~GHz Intel i3-2120 processor. Since $w=\omega(q-1)$ this covers all cases that must be
tested. It took about $7$ hours to confirm that the following odd values of $q$ are the only odd
ones for which the finite field $\mathbb{F}_q$ does not have three consecutive primitive elements:
$3$, $5$, $7$, $3^2$, $13$, $5^2$, $29$, $61$, $3^4$, $11^2$, and $13^2$. For each value of
$\omega(q-1)$ Table~\ref{wine} presents the value of $M$ that was used (second column), the number
of $m+1$ values that required testing (third column), the number of $m+1$ values that survived an
application of Theorem~\ref{main} (fourth column), the number of these that were actually primes
(fifth column, $1804641$ in total), and the number of these that were actually prime powers (sixth column, $1411$ in total).

\begin{table}[ht]
  \caption{Bounds and number of tests performed when $n=3$}
  \label{wine}
  \centering
  \begin{tabular}{rrrrrr}
    \hline\hline
    $\omega(q-1)$ & $q$ upper bound ($M$) & $m+1$ tests & $m+1$ survivors & $p$ tests & $q$ tests \\
    \hline
    $ 1$ & $                   256$ & $          7$ & $        7$ & $       3$ & $  1$ \\
    $ 2$ & $                 16384$ & $       2425$ & $      805$ & $     164$ & $  8$ \\
    $ 3$ & $              8\,02816$ & $   1\,72827$ & $    21350$ & $    4785$ & $ 26$ \\
    $ 4$ & $            317\,19424$ & $  54\,59954$ & $ 1\,49265$ & $   33357$ & $106$ \\
    $ 5$ & $           3682\,12715$ & $ 307\,38304$ & $ 6\,95172$ & $1\,59618$ & $236$ \\
    $ 6$ & $          97774\,32663$ & $2785\,78984$ & $16\,80653$ & $3\,80984$ & $405$ \\
    $ 7$ & $       4\,89130\,46416$ & $2621\,82675$ & $21\,31439$ & $4\,78146$ & $353$ \\
    $ 8$ & $      32\,73635\,05978$ & $2182\,09768$ & $21\,62062$ & $4\,76569$ & $203$ \\
    $ 9$ & $     624\,54297\,09655$ & $4790\,05331$ & $ 8\,97028$ & $1\,94276$ & $ 63$ \\
    $10$ & $    2205\,39992\,60750$ & $ 687\,95792$ & $ 2\,62534$ & $   55943$ & $  9$ \\
    $11$ & $   11712\,18570\,96884$ & $  92\,50747$ & $    93920$ & $   19315$ & $  1$ \\
    $12$ & $1\,30704\,25885\,23590$ & $  23\,78985$ & $     6566$ & $    1294$ & $  0$ \\
    $13$ & $3\,48913\,59578\,26319$ & $      11547$ & $      964$ & $     187$ & $  0$ \\
    \hline
  \end{tabular}
\end{table}

\section {Conjectures}\label{conj}

Based on numerical experiments up to $10^8$, the following conjectures appear to be plausible.

\begin{conj}\label{C1}
  The finite field $\mathbb{F}_q$ has $4$ consecutive primitive elements except when
  $q$ is divisible by $2$ or by $3$, or when $q$ is one of the following: $5$, $7$, $11$, $13$,
  $17$, $19$, $23$, $5^2$, $29$, $31$, $41$, $43$, $61$, $67$, $71$, $73$, $79$, $113$, $11^2$,
  $13^2$, $181$, $199$, $337$, $19^2$, $397$, $23^2$, $571$, $1093$, $1381$, $7^4=2401$.
\end{conj}

At present this conjecture is very difficult to settle by computation: there are simply too many
cases to test. Consider, for example, $\omega(q-1)=12$. Even though, according to Table~\ref{Surletable} we need
to go up only to $\omega(q-1)=23$, the hard cases are the intermediate values of $\omega(q-1)$. It is necessary to
test values of $q$ up to about $4\times 10^{21}$ that can have a prime power factor up to about
$2\times 10^{10}$. Given the large sizes of these numbers, the small savings gained from just
considering $q\equiv 3 \pmod 4$ and Theorem \ref{main1} will not be sufficient.
%
%
%

\begin{conj}\label{C2}
  The finite field $\mathbb{F}_q$ has $5$ consecutive primitive elements except when
  $q$ is divisible by $2$ or by $3$, or when $q$ is one of the following: $5$, $7$, $11$, $13$,
  $17$, $19$, $23$, $5^2$, $29$, $31$, $37$, $41$, $43$, $47$, $7^2$, $61$, $67$, $71$, $73$,
  $79$, $101$, $109$, $113$, $11^2$, $5^3$, $127$, $131$, $139$, $151$, $157$, $163$, $13^2$,
  $181$, $193$, $199$, $211$, $229$, $241$, $271$, $277$, $281$, $17^2$, $307$, $313$, $331$,
  $337$, $19^2$, $379$, $397$, $433$, $439$, $461$, $463$, $23^2$, $547$, $571$, $577$, $601$,
  $613$, $5^4$, $631$, $691$, $751$, $757$, $29^2$, $31^2$, $1009$, $1021$, $1033$, $1051$,
  $1093$, $1201$, $1297$, $1321$, $1381$, $1453$, $1471$, $1489$, $1531$, $1597$, $1621$, $1723$,
  $1741$, $1831$, $43^2$, $1861$, $1933$, $2017$, $2161$, $2221$, $2311$, $2341$, $7^4$, $3061$,
  $59^2$, $3541$, $3571$, $61^2$, $4201$, $4561$, $4789$, $4831$, $71^2$, $5281$, $5881$, $89^2$,
  $8821$, $9091$, $9241$, $113^2$, $5^6=15625$.
\end{conj}

We also examined fields which do not have $6$, $7$ and $8$ consecutive primitive elements. Since there
are many of these we do not list them as in Conjectures~\ref{C1} and~\ref{C2} but merely indicate
the last one we found.

\begin{conj}\label{C3}
  The finite field $\mathbb{F}_q$ has $6$ consecutive primitive elements when $q$ is not
  divisible by $2$, by $3$, or by $5$, and when $q>65521$.
\end{conj}

\begin{conj}\label{C4}
  The finite field $\mathbb{F}_q$ has $7$ consecutive primitive elements when $q$ is not
  divisible by $2$, by $3$, or by $5$, and when $q>1037401$.
\end{conj}

\begin{conj}\label{C5}
  The finite field $\mathbb{F}_q$ has $8$ consecutive primitive elements when $q$ is not
  divisible by $2$, by $3$, by $5$ or by $7$, and when $q>4476781$.
\end{conj}

\bibliographystyle{amsplain}
\bibliography{TriplePrimitive}

\providecommand{\bysame}{\leavevmode\hbox to3em{\hrulefill}\thinspace}
\providecommand{\MR}{\relax\ifhmode\unskip\space\fi MR }
\providecommand{\MRhref}[2]{%
  \href{http://www.ams.org/mathscinet-getitem?mr=#1}{#2}
}
\providecommand{\href}[2]{#2}
\begin{thebibliography}{10}

\bibitem{Carlitz}
L.~Carlitz, \emph{Sets of primitive roots}, Compositio Math. \textbf{13}
  (1956), 65--70.

\bibitem{Cohen2}
S.~D. Cohen, \emph{Consecutive primitive roots in a finite field}, Proc. Amer.
  Math. Soc. \textbf{93} (1985), no.~2, 189--197.

\bibitem{Cohen:Last}
S.~D. Cohen, \emph{Consecutive primitive roots in a finite field. {II}}, Proc.
  Amer. Math. Soc. \textbf{94} (1985), no.~4, 605--611.

\bibitem{Cohen:1985}
\bysame, \emph{Pairs of primitive roots}, Mathematika \textbf{32} (1985),
  276--285.

\bibitem{Cohen:2014}
S.~D. Cohen, T.~Oliveira e~Silva, and T.~S. Trudgian, \emph{A proof of the
  conjecture of {C}ohen and {M}ullen on sums of primitive roots}, Math. Comp.
  (2014), Accepted for publication. See also arXiv:1402.2724 [math.NT].

\bibitem{Lidl:1997}
R.~Lidl and H.~Niederreiter, \emph{Finite fields}, 2nd ed., Encyclopedia of
  mathematics and its applications, vol.~20, Cambridge University Press,
  Cambridge, 1997.

\bibitem{Peng:2014}
C.~Peng, Y.~Shen, Y.~Zhu, and C.~Liu, \emph{A note on {W}eil's multiplicative
  character sum}, Finite Fields Appl. \textbf{35} (2014), 132--133.

\bibitem{Robin}
G.~Robin, \emph{Estimation de la fonction de {T}chebychef $\theta$ sur le
  $k$-i\`{e}me nombre premier et grandes valeurs de la fonction $\omega(n)$
  nombre de diviseurs premiers de $n$}, Acta Arith. \textbf{42} (1983), no.~4,
  367--389.

\bibitem{Tanti}
J.~Tanti and R.~Thangadurai, \emph{Distribution of residues and primitive
  roots}, Proc. Indian Acad. Sci. (Math. Sci.) \textbf{123} (2013), no.~2,
  203--211.

\bibitem{pari}
{The PARI~Group}, Bordeaux, \emph{{PARI/GP version {\tt 2.7.2}}}, 2014,
  available from \url{http://pari.math.u-bordeaux.fr/}.

\end{thebibliography}

\end{document}